\documentclass{amsart}

\usepackage{graphicx}
\usepackage{amssymb}

\newtheorem{theorem}{Theorem}[section]
\newtheorem{introtheorem}{Theorem}
\newtheorem{proposition}[theorem]{Proposition}
\newtheorem{introproposition}[introtheorem]{Proposition}
\newtheorem{definition}[theorem]{Definition}

\newtheorem{corollary}[theorem]{Corollary}
\newtheorem{lemma}[theorem]{Lemma}

\theoremstyle{definition}

\numberwithin{equation}{section}

\renewcommand{\bold}[1]{\medskip \noindent {\bf #1 }\nopagebreak}

\DeclareMathOperator{\Aut}{Aut}
\DeclareMathOperator{\Sym}{Sym}
\DeclareMathOperator{\Homeo}{Homeo}
\DeclareMathOperator{\Comm}{Comm}

\DeclareMathOperator{\comIndex}{c}

\DeclareMathOperator{\Z}{\mathbb{Z}}

\newcommand{\comGrowth}{{\bf c}}
\DeclareMathOperator{\stabilizer}{stab}
\newcommand{\addingMachine}{\mathcal{A}}

\newcommand{\setU}

\def\Z{\mathbb{Z}}
\def\N{\mathbb{N}}

\providecommand{\bysame}{\leavevmode\hbox to3em{\hrulefill}\thinspace}
\providecommand{\MR}{\relax\ifhmode\unskip\space\fi MR }
\providecommand{\MRhref}[2]{%
  \href{http://www.ams.org/mathscinet-getitem?mr=#1}{#2}
}
\providecommand{\href}[2]{#2}

\begin{document}

\title{Commensurability growth of branch groups}

\author{Khalid Bou-Rabee}
\address{Department of Mathematics, The City College of New York}
\curraddr{Department of Mathematics,
Convent Ave at 138th Street, New York, NY, 10031}
\email{khalid.math@gmail.com}
\thanks{KB supported in part by NSF Grant \#1405609.}

\author{Rachel Skipper}
\address{Binghamton University}
\email{skipper@math.binghamton.edu}

\author{Daniel Studenmund}
\address{University of Notre Dame}
\email{dstudenm@nd.edu}
\thanks{DS supported in part by NSF Grant \#1547292.}


\date{August 26, 2018.}



\begin{abstract}
Fixing a subgroup $\Gamma$ in a group $G$, the commensurability growth function assigns to each $n$ the cardinality of the set of subgroups $\Delta$ of $G$ with $[\Gamma: \Gamma \cap \Delta][\Delta : \Gamma \cap \Delta] = n$. 
For pairs $\Gamma \leq A$, where $A$ is the automorphism group of a $p$-regular tree and $\Gamma$ is finitely generated, we show that this function can take on finite, countable, or uncountable cardinals.
For almost all known branch groups $\Gamma$ (the first Grigorchuk group, the twisted twin Grigorchuk group, Pervova groups, Gupta-Sidki groups, etc.) acting on $p$-regular trees, this function is precisely $\aleph_0$ for any $n = p^k$.
\end{abstract}

\maketitle

\section*{Introduction}

Two subgroups $\Delta_1$ and $\Delta_2$ of a group $G$ are
\emph{commensurable} if their \emph{commensurability index} 
\[
  \comIndex(\Delta_1, \Delta_2) := [\Delta_1 : \Delta_1 \cap
  \Delta_2][\Delta_2 : \Delta_1 \cap \Delta_2]
\] 
is finite.
For a pair of groups $\Gamma \leq G$, the {\em commensurability growth function} $\N \to \N \cup \{ \infty \}$ assigns to each $n \in \N$ the cardinality
 $$\comGrowth_n(\Gamma,G) := | \{ \Delta \leq G : \comIndex\left(\Gamma, \Delta \right) = n \} |.$$
This function was first systematically studied in \cite{BD18}, where it was used to give regularity results on the structure of arithmetic lattices in a unipotent algebraic group.
Here, we continue this study to pairs of groups naturally arising from the class of finitely generated residually finite groups.
This extends the study of commensurability growth beyond the class of linear groups.

Associated to any residually finite group $\Gamma$ are many rooted finite-valent trees $T$ where $\Gamma \leq \Aut(T)$, the automorphism group of $T$.
Such pairs are particularly beautiful and useful when the rooted tree is $d$-regular, denoted $T_d$, and the subgroup is \emph{branch}. For instance, the \emph{first Grigorchuk group}, a branch subgroup of $\Aut(T_2)$, has intermediate growth \cite{MR764305}, is commensurable with its direct product, and is a counter-example to the Burnside Problem \cite{MR1786869}.
\emph{To what extent does the sequence $\{\comGrowth_n(\Gamma, \Aut(T_d))\}_{n=1}^\infty$ distinguish branch subgroups of $\Aut(T_d)$ among the collection of subgroups of $\Aut(T_d)$?}

A simple example of a non-branch subgroup of $\Aut(T_d)$ is the embedding of $\Z$ into $\Aut(T_2)$ known as the \emph{binary adding machine}. See \S \ref{sec:proofaddingmachine} for the definition.
\begin{introproposition} \label{prop:addingmachine}
Let $\addingMachine$ be the \emph{binary adding machine subgroup} of $\Aut(T_2)$.
Then for every natural number $k$, $\comGrowth_{2^k}(\addingMachine, \Aut(T_2)) = \aleph_1$. On the other hand, there exists an infinite dihedral group $H \leq \Aut(T_2)$, containing $\addingMachine$ as a subgroup of index two such that
$\comGrowth_2(H, \Aut(T_2)) = 3$.
\end{introproposition}

Our proof of Proposition \ref{prop:addingmachine}, given in \S \ref{sec:proofaddingmachine}, uses results from \cite{MR2171236}.
For every natural number $n$, the group $\addingMachine$ acts transitively on vertices of distance $n$ from the root in $T_2$. Thus, while $\addingMachine$ in some sense fills up $\Aut(T_2)$, there are many subgroups of finite commensurability index with $\addingMachine$.

In contrast to this behavior, our main result shows that most well-studied examples of branch groups sitting inside the Sylow pro-$p$ subgroup of $\Aut(T_p)$, where $p$ is a prime, have the same commensurability growth values.
These examples include the first Grigorchuk group, the twisted twin of the Grigorchuk group, the Pervova groups, the Gupta-Sidki $p$-groups, the Fabrykowski-Gupta group and an infinite family of generalizations of the Fabrykowski-Gupta group, and GGS groups with non-constant accompanying vector. 
Important to our proof is that all these examples satisfy the \emph{rigid congruence subgroup property}, a weakening of the usual congruence subgroup property. See \S \ref{sec:preliminaries} for definitions of these groups and their properties.

\begin{introtheorem} \label{thm:main}
Let $\Gamma$ be a finitely generated, self-similar, regular branch group over a branching subgroup $K$ in $\Aut(T_p)$. Suppose $\Gamma$ is contained in the Sylow pro-$p$ subgroup of $\Aut(T_p)$ and satisfies the rigid congruence subgroup property.
Then $\comGrowth_{p^k}(\Gamma, \Aut(T_p)) = \aleph_0$ for all $k$.
\end{introtheorem}
\noindent
R\"over's theorem \cite{rov02} on abstract commensurators of branch groups is key to the proof of the upper bound $\comGrowth_{p^k}(\Gamma, \Aut(T_p)) \leq \aleph_0$ (see \S \ref{sec:mainproof}).

Commensurability growth is a generalization of subgroup growth to pairs of groups \cite{MR1978431}. While the subgroup growth function of a finitely generated group is always finite, this paper gives the first naturally occurring pairs where the commensurability growth function has infinite values.


\subsection*{Acknowledgements}
We are grateful to Benson Farb, Andrew Putman, Benjamin Steinberg, and Slobodan Tanushevski for their conversations and support. We thank Benson Farb for helpful comments on an earlier draft.

\section{Preliminaries} \label{sec:preliminaries}

The groups we shall consider will all be subgroups of the group $\Aut(T_d)$ of automorphisms of a $d$-regular rooted tree $T_d$. We will always consider our trees $T_d$ to arise from the following construction. 
Let $X$ be a finite alphabet with $|X|=d\geq 2$ and given a fixed total ordering. The vertex set of the tree $T_X$ is the set of finite sequences over $X$; two sequences are connected by an edge when one can be obtained from the other by right-adjunction of a letter in $X$. The root is the empty sequence $\emptyset$, and the children of $v$ are all $v x$ for $x\in X$. The length of a sequence $v$ is denoted by $|v|$. The set $X^n\subset T_X$, of all sequences of length $n$, is called the \emph{$n$th level} of the tree $T_X$. 

\subsection{Automorphisms}

Let $g$ be an automorphism of the rooted tree $T_X$. 
For a vertex $v\in T_X$, consider the rooted subtrees $vT_X=\{vw\mid w\in T_X\}$ and $g(v)T_X=\{g(v)w\mid w\in T_X\}$ with roots $v$ and $g(v)$ respectively. 
Notice that the map $vT_X\to g(v)T_X$, given by $vw \mapsto g(v)w,$ is a morphism of rooted trees. 
Moreover, the subtrees $vT_X$ and $g(v)T_X$ are naturally isomorphic to $T_X$. Identifying $vT_X$ and $g(v)T_X$ with $T_X$ we get an automorphism $g|_v\colon T_X\to T_X$ uniquely defined by the condition 
$$
g(vw)=g(v)g|_v(w)
$$ for all $w\in T_X$.
We call the automorphism $g|_v$ the \emph{section} of $g$ at $v$. Observe the following obvious properties of the sections:
\begin{align*}
g|_{v_1 v_2}&=g|_{v_1}|_{v_2}\\
(g_1\cdot g_2)|_v&=g_1|_{g_2(v)}\cdot g_2|_v.
\end{align*}

It follows that the action of the automorphism $g \in \Aut(T_X)$  can be written as $g = (g_1,\dots, g_{|X|}) \pi_g$, where $\pi_g \in \Sym(X)$ is the permutation defined by the action of $g$ on the first level of the tree, and $g_1,\dots, g_{|X|} \in \Aut(T_X)$ are the sections of $g$ at the vertices of the first level of $T_X$. This gives an isomorphism $\Aut(T_X)\cong \Aut(T_X)\wr \Sym(X)$.

For $H\leq \Aut(T_X)$, we write $X\star H$ to indicate the group of automorphisms $g$ with $\pi_g=1$ and $g_i\in H$ for all $1\leq i \leq |X|$. Similarly, let $X^n\star H$ indicate the group of elements $g\in \Aut(T_X)$ with $\pi_{g|_v}=1$ for all $v$ on level less than $n$ and $g|_u \in H$ for all $u$ on the $n$th level. 

\subsection{Self-similar and branch groups}

A subgroup $G$ of $\Aut(T_X)$ is \emph{self-similar} if for every $g\in G$ and every $v\in T_X$ the section $g|_v \in G$. For example, the full automorphism group $\Aut(T_X)$ is itself self-similar.

Let $G\leq \Aut(T_X)$ be a group of automorphisms of the rooted tree $T_X$. 
For a vertex $v\in T_X$ the \emph{vertex stabilizer} is the subgroup consisting of the automorphisms that fix the sequence $v$:
$$\stabilizer_G(v)=\{g\in G\mid g(v)=v\}.$$ 
The \emph{$n$th level stabilizer} (also called the \emph{$n$th principal congruence subgroup}) is the subgroup $\stabilizer_G(n)$ consisting of the automorphisms that fix all vertices of the $n$th level:  $$\stabilizer_G(n)=\cap_{v\in X^n} \stabilizer_G(v).$$
Stabilizer subgroups $\stabilizer_G(n)$ with $n\geq 0$ are normal in $G$. 

Notice that any $g\in \stabilizer_G(n)$ can be identified in a natural way with the sequence of sections at vertices in $X^n$ \[(g_1, \dots, g_{|X|^n})\]
taken in the lexicographical ordering on $X^n$. We say that $g$ is \emph{of level $n$} if $g\in \stabilizer_G(n)\setminus \stabilizer_G(n+1)$. 

The \emph{rigid stabilizer} $\operatorname{rist}_G(v)$ of a vertex $v\in T_X$ is the subgroup of $G$ of all automorphisms acting non-trivially only on the vertices of the form $vu$ with $u\in T_X$: $$\operatorname{rist}_G(v)=\{g\in G\mid g(w)=w \text{ for all } w\notin vT_X\}$$
The \emph{$n$th level rigid stabilizer} $$\operatorname{rist}_G(n)=\langle \operatorname{rist}_G(v)\mid v\in X^n\rangle$$ is the subgroup generated by the union of the rigid stabilizers of the vertices of the $n$th level.

We say that a subgroup $G \leq \Aut(T_X)$ is {\em level-transitive} if $G$ acts transitively on each level of $T_X$. An automorphism $g$ is {\em level transitive} if $\langle g \rangle$ is level-transitive.  A level-transitive subgroup $G\leq Aut(T_X)$ is \emph{branch} if $\operatorname{rist}_G(n)$ is of finite index in $G$ for all $n \geq 1$. In this article we will restrict ourselves to the particularly important type of branch groups introduced by the following definition. 

\begin{definition}
A level-transitive group $G\leq \Aut(T_X)$ is \emph{regular branch} if there exists a finite-index subgroup $K$ of $G$ such that $K$ contains $X\star K$ of finite index. In this case, $K$ is called a \emph{branching subgroup} for $G$. Call $G$ \emph{layered} if $G$ itself is a branching subgroup for $G$.
\end{definition}

A subgroup $G$ of $\Aut(T_X)$ is said to satisfy the \emph{congruence subgroup property} if any finite index subgroup $H$ of $G$ contains a principal congruence subgroup $\stabilizer_G(n)$ for some $n\geq 1$. 

\begin{definition} \label{defn:RCSP}
A subgroup $G$ of $\Aut(T_X)$ has the \emph{rigid congruence subgroup property} if every level rigid stabilizer of $G$ contains a level stabilizer of $G$.
\end{definition}

\subsection{The Sylow pro-$p$ subgroup}
\label{sec:sylow}
$\Aut(T_d)$ is a profinite group; it is canonically isomorphic to $\underset{n \geq 1}{\varprojlim} \Aut(T_d(n))$ where $T_d(n)$ is the finite subtree of $T_d$ consisting of vertices of level less than or equal to $n$.

In the case that $d = p$ for a prime $p$, fix a cyclic permutation $\sigma \in \Sym(X)$ of order $p$. The {\em Sylow pro-$p$ subgroup} $\Aut_p(T_p) \leq \Aut(T_p)$ consists of automorphisms $g \in \Aut(T_p)$ such that at every vertex $v\in X^*$ the section $g|_v$ acts on $X$ as $\sigma^i$ for some $0\leq i \leq p-1$ (see \cite{MR1765119} pages 133-134). For a self-similar group $G \leq \Aut_p(T_p)$ we have the containment $G\leq G\wr \langle \sigma \rangle$ under the isomorphism $\Aut_p(T_p) \cong \Aut_p(T_p)\wr \left\langle \sigma \right\rangle$. If $G$ is layered, there is an inclusion $X \star G \leq G$. Since a layered group is level-transitive, it follows that a self-similar and layered subgroup $G\leq \Aut_p(T_p)$ satisfies $G\cong G\wr \left\langle \sigma \right\rangle$.

\subsection{Examples} \label{subsec:examples}

The following examples are self-similar regular branch groups with the rigid congruence subgroup property.

\bold{The First Grigorchuk group:}
Let $X=\{1,2\}$. 
Define automorphisms of $T_X$ inductively by
$$
a = \sigma,\; b = (a,c),\; c = (a,d), \text{ and } d = (1,b),
$$
where $\sigma$ is the transposition $(1,2)\in Sym(X)$.
The \emph{first Grigorchuk group} is $\Gamma:=\langle a, b, c, d \rangle$. Clearly, $\Gamma$ is self-similar. Moreover, it is regular branch \cite{MR764305} over the subgroup 
\[  K=\langle (ab)^2, (bada)^2, (abad)^2\rangle.\]
It also has the congruence subgroup property \cite[Proposition 10]{MR1765119}.

\bold{The Twisted Twin:} Let $X=\{1,2\}$. Define automorphisms of $T_X$ inductively by
$$
a=\sigma, \; \beta=(\gamma, a), \; \gamma=(a, \delta), \text{ and } \delta=(1, \beta).
$$
The {\em Twisted Twin of the Grigorchuk group} is $G := \langle a, \beta, \gamma, \delta \rangle$. It is a self-similar regular branch group \cite{bs10} with branching subgroup 
\[K=\langle \langle [a,\beta], [\beta, \gamma], [\beta, \delta], [\gamma, \delta], \beta\delta\gamma \rangle \rangle ^G.\]
It does not have the congruence subgroup property but does have the rigid congruence subgroup property \cite{BSZ12}.

\bold{Gupta-Sidki groups:} Let $X=\{1,\dots, p\}$ where $p$ is odd prime.
Define automorphisms $x$ and $y$ of $T_X$ inductively by:
$$x= \sigma \text{ and } y =(x,x^{-1}, 1,\dots, 1,y ),$$
where $\sigma$ is the cyclic permutation $(1,2,\dots p)$ on $X$. 
The {\em Gupta-Sidki $p$-group} is  $G_p := \langle x, y \rangle$. Clearly, $G_p$ is self-similar. It is regular branch over its commutator subgroup \cite{MR759409, MR767112}. Moreover, $G_p$ has the congruence subgroup property (see \cite[Proposition 2.6]{MR3119213}).

\bold{Gupta-Sidki variations:} There are various modifications of the Gupta-Sidki group which are self-similar, regular branch groups having the congruence subgroup property. Here is an example of such a modification. Let $G$ be the subgroup of automorphisms on the rooted $p$-regular tree for $p\geq 7$ generated by $x=(1,2,\dots, p)$ and $y=(x^{i_1}, x^{i_2}, \dots, x^{i_p-3},1,1,1,y)$ for $0\leq i_j\leq p-1$ and $i_1\neq 0$. The group $G$ is regular branch over its commutator subgroup (see \cite[Example $10.2$]{MR1765119}).

\bold{Fabrykowski-Gupta group:} Let $X = \{ 1, 2, 3\}$. Define automorphisms of $T_X$ inductively by
\[
    a=(1,2,3) \text{ and }  b=(a,1,b).
\]
The {\em Fabrykowski-Gupta group} is $\mathcal{G} := \langle a,b \rangle$. Then $\mathcal{G}$ is a regular branch group with the congruence subgroup property (see \cite[6.2, 6.4]{MR1899368}). 

A natural generalization of the Fabrykowski-Gupta group is a group $\mathcal{G}_p$ generated by automorphisms $a=(1,2,\dots, p)$ and $b=(a,1,\dots, 1, b)$ of a $p$-regular tree. For every prime $p\geq 5$, the Fabrykowski-Gupta group $\mathcal{G}_p$ is regular branch with the congruence subgroup property \cite[Example 10.1]{MR1765119}.

\bold{EGS groups:} Let $X = \{1,  \dotsc, p\}$ and let $\bar \iota = (i_1, i_2, \dots, i_{p-1})$ be a non-symmetric vector of integers between $0$ and $p-1$, so that $i_j \neq i_{p-j}$ for some $j$. Define automorphisms of $T_X$ inductively by
\[
a=\sigma, \; b=(a^{i_1}, a^{i_2}, \dots, a^{i_{p-1}}, b), \text{ and } c=(c, a^{i_1}, a^{i_2}, \dots, a^{i_{p-1}})
\]
where $\sigma$ is the permutation $(1,2, \dots, p)$.
The {\em extended Gupta-Sidki (EGS) group} is $\Gamma_{\bar \iota} := \langle a,b,c \rangle$.
Pervova constructed the EGS groups as the first examples of branch groups failing to have the congruence subgroup property \cite{per07}. It was shown in \cite{BSZ12} that these groups nevertheless do satisfy the rigid congruence subgroup property.
These groups are clearly self-similar and moreover are regular branch groups having their commutator subgroup as a branching subgroup \cite{per07}.

\section{The adding machine: Proof of Proposition \ref{prop:addingmachine}}

\label{sec:proofaddingmachine}

The {\em binary adding machine} $\addingMachine \leq \Aut(T_2)$ is the infinite cyclic subgroup generated by $\tau := (1, \tau)\sigma$.

\begin{proposition} \label{prop:addingmachine1}
For every natural number $k$, $\comGrowth_{2^k}(\addingMachine, \Aut(T_2)) = \aleph_1$.
\end{proposition}

\begin{proof}
Note that it suffices to show that $\comGrowth_2(\addingMachine, \Aut(T_2)) \geq \aleph_1$, since the cardinal of the collection of all finitely generated subgroups of $\Aut(T_2)$ is $\aleph_1$.
In Theorem 4.13 from \cite{MR2171236}, it is shown that $\langle \tau \rangle$ is normalized by elements of the form
$\tau^x u_y$ where $y$ is an odd integer, $x$ is a 2-adic integer, and $u_y := (u_y, u_y \tau^{(y-1)/2}).$
Notice that for any fixed $x\in \Z_2$, the element $\tau^x u_{-1}$ has order two and normalizes $\addingMachine$, and hence $\langle \tau^x u_{-1}, \addingMachine\rangle$ contains $\addingMachine$ as a subgroup of index two.
Moreover, since $\tau^x u_{-1}$ has order two and $u_{-1} \tau u_{-1} = \tau^{-1}$, we have set equalities
$$\langle \tau^x u_{-1}, \addingMachine \rangle
= \langle \tau^x u_{-1}, \tau \rangle
=
\{ \tau^{x+n} u_{-1} : n \in \Z \} \sqcup \addingMachine.$$
Since canonically, $\{ \tau^x : x \in \Z_2\} \cong \Z_2$, it follows that the cardinality of all such sets $\langle \tau^x u_{-1}, \addingMachine \rangle$ as $x$ varies over $\Z_2^*$ is equal to the cardinality of $\Z_2^*/\Z$, which is $\aleph_1$, as desired.
\end{proof}

Now, any two level-transitive automorphisms in $\Aut(T_d)$ are conjugate in $\Aut(T_d)$ (see \cite{GNS01} Corollary 4.1). Since $\addingMachine$ is clearly level-transitive, we get the following immediate corollary.

\begin{corollary}
Let $g\in \Aut(T_2)$ be a level-transitive automorphism. Then for every natural number $k$, $\comGrowth_{2^k}(\langle g \rangle, \Aut(T_2))=\aleph_1$.
\end{corollary}

We now prove the second half of Proposition \ref{prop:addingmachine}. Fix the element $\delta = (\delta, \delta) \sigma$ and set $H = \langle \delta, \tau \rangle$.
We use the following theorem from \cite{MR2171236}.

\begin{theorem}[Theorem 4.12 \cite{MR2171236}]
The group $H$ is infinite dihedral.
Moreover, $H$ is its own normalizer in $\Aut(T_2)$.
\end{theorem}

\begin{proposition} \label{prop:addingmachine2}
The group $H$ satsifies
$\comGrowth_2(H, \Aut(T_2)) = 3$.
\end{proposition}

\begin{proof}
Notice that if $H_0 \leq \Aut(T_2)$ contains $H$ with $[H_0 : H] = 2$, then $H \lhd H_0$ and so $H_0$ is contained in the normalizer of $H$.
Hence, there does not exist a supergroup $H_0$ containing $H$ as a subgroup of index two.
Moreover, since $H$ is infinite dihedral, there are only three subgroups of $H$ of index two, and so $\comGrowth_2(H, \Aut(T_2) = 3$ as desired.
\end{proof}

\section{Branch groups: Proof of Theorem \ref{thm:main}}
\label{sec:mainproof}

\begin{lemma}
\label{lem:notlayered}
Let $\Gamma \leq Aut_p(T_p)$ be self-similar and finitely generated. Then $\Gamma$ is not layered.
\end{lemma}
\begin{proof}
Suppose that $\Gamma$ is layered. We will show that it can not be finitely generated. Indeed, let $C_p$ be a cyclic group of order $p$ and for each $i$ define a homomorphism $\psi_i:\Gamma \rightarrow C_p$ by $\prod_{|v|=i} \pi_{g|_v}$. Now for each $n$ let $\Psi_n = \prod_{i=0}^n \psi_i: \Gamma \rightarrow \bigoplus_{i=0}^n C_p$. Since $\Gamma$ is layered, the remarks of \S\ref{sec:sylow} give an isomorphism $\Gamma  \cong \Gamma\wr\langle \sigma \rangle$ under the isomorphism $\Aut_p(T_p)\cong \Aut_p(T_p)\wr \left\langle \sigma \right\rangle$. This implies that $\Psi_n$ is surjective for each $n$. Since the groups $\bigoplus_{i=0}^n C_p$ require arbitrarily many generators as $n$ tends to infinity, the group $\Gamma$ is not finitely generated.
\end{proof}

We now discuss some consequences of the rigid congruence subgroup property. For a regular branch group $\Gamma$ with maximal branching subgroup $K$, Corollary 1.6 in \cite{BSZ12} says $K$ contains a level rigid stabilizer. Consequently, if $\Gamma$ has the rigid congruence subgroup property, $K$ also contains a level stabilizer. In particular, there exists an $m$ with $\stabilizer_{\Gamma}(m)\leq K$.

\begin{lemma}\label{lembranchingstab}
Let $\Gamma$ be a self-similar regular branch group with maximal branching subgroup $K$ and with the rigid congruence subgroup property. Then for all $n\geq 0$, $\stabilizer_\Gamma(m+n)=X^n\star \stabilizer_\Gamma(m)$ where $m$ is such that $\stabilizer_\Gamma(m)\leq K$.
\end{lemma}

\begin{proof}
Let $\Gamma$ be a self-similar regular branch group with the rigid congruence subgroup property and let $K$ be the maximal branching subgroup for $\Gamma$. Since $\Gamma$ is self-similar, $\stabilizer_\Gamma(m+n)\leq X^n\star\stabilizer_\Gamma(m)$ for all $m$ and $n$.

Now let $m$ be such that $\stabilizer_\Gamma(m)\leq K$. As $K$ is a branching subgroup, for all $n\geq 0$, $X^n\star K\leq K$ and so we get the following set of inclusions:
\[X^n\star \stabilizer_{\Gamma}(m)\leq X^n\star K \leq K \leq \Gamma\]
and so $X^n\star \stabilizer_{\Gamma}(m)$ is contained in $\Gamma$ and stabilizes level $(m+n)$. Thus we conclude $\stabilizer_\Gamma(m+n)= X^n\star \stabilizer_\Gamma(m)$ as desired.
\end{proof}

Let $Q_n=\{g\in \Aut_p(T_p) \mid g|_v=1 \text{ for all $v$ with } |v| \geq n \}$. Observe that the group $(X^n\star\Gamma)\cap Q_n=\{1\}$ and so $\langle X^n\star \Gamma, Q_n \rangle = (X^n \star \Gamma)\rtimes Q_n$.

\begin{lemma}\label{lem:samestabs}
Let $\Gamma\in \Aut_p(T_p)$ be a self-similar regular branch group with maximal branching subgroup $K$ and the rigid congruence subgroup property. Let $n\geq 0$. Then for all sufficiently large $k$, $\stabilizer_{\Gamma}(k)=\stabilizer_{(X^n\star \Gamma)\rtimes Q_n}(k)$.
\end{lemma}

\begin{proof}
Note that it suffices to show equality for a fixed $k$, as the stabilizer of the $k+1$ level is precisely the set of elements in the stabilizer of level $k$ which also stabilize level $k+1$.

Let $k=n+m$ where $m$ is such that $\stabilizer_{\Gamma}(m)\leq K$. As $\Gamma$ is self-similar, $\Gamma\leq (X^n\star \Gamma)\rtimes Q_n$ and so similarly $\stabilizer_{\Gamma}(k)\leq\stabilizer_{(X^n\star\Gamma)\rtimes Q_n}(k)$.

For the other inclusion, note that only the identity element in $Q_n$ stabilizes level $n+m$ and so $\stabilizer_{(X^n\star \Gamma)\rtimes Q_n}(n+m)=\stabilizer_{(X^n\star \Gamma)}(n+m)$. Moreover, we have
$\stabilizer_{(X^n\star\Gamma)}(n+m)=X^n\star \stabilizer_{\Gamma}(m)$, which by Lemma~\ref{lembranchingstab} is precisely equal to  $\stabilizer_{\Gamma}(n+m)$.
\end{proof}

We now establish an upper bound on the commensurabilty growth. Our proof uses the abstract and relative commensurators. The {\em relative commensurator} of a subgroup $H$ in a group $G$ is
\[
\Comm_G(H) := \left\{ g\in G \mid \comIndex(g H g^{-1}, H) < \infty \right\}.
\]
The {\em abstract commensurator} of $G$ is the set of equivalence classes of isomorphisms $\phi: H_1 \to H_2$ for finite-index subgroups $H_1,H_2\leq G$, where two isomorphisms are equivalent if they are both defined and equal on a common finite-index subgroup of $G$.

\begin{proposition}
\label{prop:upperbound}
Let $\Gamma$ be a finitely generated, self-similar, regular branch group over a branching subgroup $K$ in $\Aut(T_p)$. Suppose $\Gamma$ is contained in the Sylow pro-$p$ subgroup of $\Aut(T_p)$ and satisfies the rigid congruence subgroup property.
Then \[\comGrowth_{p^k}(\Gamma, \Aut(T_p)) \leq \aleph_0.\]
\end{proposition}

\begin{proof}
Consider the map $\Phi: \Aut(T_p) \to \Homeo(\partial T_p)$ that sends an element to its induced action on the boundary. Note that $\Phi$ is injective. If $\Delta \leq \Aut(T_p)$ satisfies $\comIndex(\Gamma, \Delta) < \infty$ then $\Phi(\Delta) \leq \Comm_{\Homeo(\partial T_p)}(\Gamma)$. Therefore, the map $\Phi$ faithfully maps the collection of subgroups of $\Aut(T_p)$ commensurable with $\Gamma$ into the collection of finitely generated subgroups of $\Comm_{\Homeo(\partial T_p)}(\Gamma)$. R\"over~\cite{rov02} has shown that $\Comm_{\Homeo(\partial T_p)}(\Gamma)$ is isomorphic to $\Comm(\Gamma)$, the abstract commensurator of $\Gamma$. Because $\Gamma$ is finitely generated, $\Comm(\Gamma)$ is countable. Therefore there are countably many finitely generated subgroups of $\Comm_{\Homeo(\partial T_p)}(\Gamma)$, and so there are countably many subgroups $\Delta \leq \Aut(T_p)$ commensurable with $\Gamma$.
\end{proof}

We finish the proof by supplying the $\aleph_0$ lower bound:

\begin{theorem} \label{thm:mainthm}
Let $\Gamma$ be a finitely generated, self-similar, regular branch group over a branching subgroup $K$ in $\Aut(T_p)$. Suppose $\Gamma$ is contained in the Sylow pro-$p$ subgroup of $\Aut(T_p)$ and satisfies the rigid congruence subgroup property.
Then $\comGrowth_{p^k}(\Gamma, \Aut(T_p)) = \aleph_0$ for all $k$.
\end{theorem}

\begin{proof}
Fix $k\geq 1$. Proposition~\ref{prop:upperbound} provides an upper bound $\comGrowth_{p^k}(\Gamma, \Aut(T_p))\leq \aleph_0$. To prove the lower bound, fix  a subgroup $H\leq \Gamma$ of index $p^{k-1}$ containing $\stabilizer_{\Gamma}(N)$ for some $N \in \N$. We will construct infinitely many index $p$ extensions of $H$ not contained in $\Gamma$. To find these extensions, we will inductively construct an infinite sequence of pairs $(\tilde{H}_i, n_i)_{i=1}^\infty$ such that
$H \leq \tilde{H}_i$ with $[\tilde{H}_i: H]=p$ and $H/ \stabilizer_{H}(n_{i-1})= \tilde{H}_i/\stabilizer_{\tilde{H}_i}(n_{i-1})$, but $H/ \stabilizer_{H}(n_i)\neq \tilde{H}_i/\stabilizer_{\tilde{H}_i}(n_i)$. 
It is immediate from the latter condition that the $\tilde{H}_i$'s are pairwise distinct.
In the case that $k=1$, in which $H = \Gamma$, this completes the proof. See below for the end of the argument in the case $k\geq 2$.

For the base case of the induction, choose $1\neq \gamma_1 \in (X\star \Gamma)\setminus \Gamma$. Such a $\gamma_1$ exists because $\Gamma$ is finitely generated and thus is not layered by Lemma~\ref{lem:notlayered}. Let $\Gamma_1=\langle \Gamma, \gamma_1 \rangle$. Then $[\Gamma_1: \Gamma]=p^{k_1}$ for some $k_1\geq 1$. To see this, let $Q_1=\langle \sigma \rangle$ be as defined above and recall that Lemma~\ref{lem:samestabs} gives an inclusion  $\stabilizer_{(X\star \Gamma)\rtimes Q_1}(n)\leq \Gamma$ for sufficiently large $n$. Therefore there is a chain of subgroups
\[\stabilizer_{(X\star \Gamma)\rtimes Q_1}(n)\leq \Gamma \leq \Gamma_1 \leq (X\star \Gamma)\rtimes Q_1\]
and the index of $\stabilizer_{(X\star \Gamma)\rtimes Q_1}(n)$ in $(X\star \Gamma)\rtimes Q_1$ is a power of $p$ as $(X\star \Gamma)\rtimes Q_1 \leq \Aut_p(T_p)$.

Now select $\tilde H_1 \leq \Gamma_1$ such that $H \leq \tilde H_1$ and $[\tilde H_1:H] = p$ where $\tilde H_1 = \langle H, \tilde h_1 \rangle$ for some $\tilde h_1$.
Now, since $\Gamma \leq \Gamma_1 \leq (X\star \Gamma)\rtimes Q_1$, by Lemma~\ref{lem:samestabs} there exists $n_1 > N$ with $\stabilizer_{\Gamma_1}(n_1)=\stabilizer_{\Gamma}(n_1)\leq \Gamma$, and therefore $\stabilizer_{\tilde H_1}(n_1) \leq \stabilizer_{\Gamma}(n_1)$.
On the other hand, since $\stabilizer_{\Gamma}(N) \leq H \leq \tilde H_1$, we clearly have $\stabilizer_{\Gamma}(n_1) \leq \stabilizer_{\tilde H_1}(n_1)$.  
Therefore $\stabilizer_{\tilde H_1}(n_1) = \stabilizer_{\Gamma}(n_1)$.
Letting $\pi_{n_1} : \tilde H_1 \to \tilde H_1 / \stabilizer_{\tilde H_1}(n_1)$, we see that $\pi_{n_1}(\tilde h_1) \notin \pi_{n_1}(H).$
This completes the base case of the induction.

Now assume for some $j$, we have built a sequence of pairs $(\tilde{H}_i, n_i)_{i=1}^{j}$ as described above. 
Choose $1\neq \gamma_{j+1}\in (X^{n_j}\star\Gamma) \setminus \Gamma$. Then, as in the argument in the base case,  $\Gamma_{j+1}=\langle \Gamma, \gamma_{j+1}\rangle$ contains $\Gamma$ as a subgroup of index $p^{k_{j+1}}$ for some $k_{j+1}$.
There exists $\tilde h_{j+1}$ such that $\tilde H_{j+1}=\langle \tilde h_{j+1}, H\rangle \leq \Gamma_{j+1}$ contains $H$ as a subgroup of index $p$. Clearly, $\tilde{H}_{j+1}/\stabilizer_{\tilde{H}_{j+1}}(n_j)=H/\stabilizer_{H}(n_j)$ as $\tilde{h}_{j+1}$ stabilizes level $n_j$.
Moreover, again as in the argument in the base case, there exists an $n_{j+1} > N$  such that 
$\stabilizer_{\tilde{H}_{j+1}}(n_{j+1})=\stabilizer_{\Gamma}(n_{j+1})$. Hence, taking 

$$
\pi_{n_{j+1}}: \tilde{H}_{n_{j+1}} \rightarrow \tilde{H}_{n_{j+1}}/\stabilizer_{\tilde H_{n_{j+1}}}(n_{j+1}),
$$
we see that $\pi_{n_{j+1}}(\tilde{h}_{n_{j+1}})\notin \pi_{n_{j+1}}(H)=H/\stabilizer_{H}(n_{j+1})$. The induction is complete.

The proof is complete in the case $k=1$, so consider now the case $k\geq 2$. Each $H_i$ constructed satisfies either $\tilde H_i\cap \Gamma = H$ or $\tilde H_i \leq \Gamma$ with $[\Gamma : \tilde H_i] = p^{k-2}$.
Since there are finitely many subgroups of $\Gamma$ of index $p^{k-2}$ and there are infinitely many distinct $\tilde H_i$ with $[\tilde H_i :H] = p$, we know that there exists an index set $S$ of cardinality $\aleph_0$ such that $\tilde H_i \cap \Gamma = H$ for every $i \in S$.
For every $i \in S$, we then have $\comIndex(\Gamma, \tilde H_i)=[\Gamma:H][\tilde H_i:H]=p^{k-1}p=p^k$, giving us the desired lower bound on $\comGrowth_{p^k}(\Gamma, \Aut(T_p))$.
\end{proof}



\begin{thebibliography}{GNS01}

\bibitem[BG02]{MR1899368}
Laurent Bartholdi and Rostislav~I. Grigorchuk, \emph{On parabolic subgroups and
  {H}ecke algebras of some fractal groups}, Serdica Math. J. \textbf{28}
  (2002), no.~1, 47--90. \MR{1899368}

\bibitem[BRS]{BD18}
Khalid Bou-Rabee and Daniel Studenmund, \emph{Arithmetic lattices in unipotent
  algebraic groups}, arXiv:1804.04973.

\bibitem[BS06]{MR2171236}
Laurent Bartholdi and Said~N. Sidki, \emph{The automorphism tower of groups
  acting on rooted trees}, Trans. Amer. Math. Soc. \textbf{358} (2006), no.~1,
  329--358. \MR{2171236}

\bibitem[BS10]{bs10}
Laurent Bartholdi and Olivier Siegenthaler, \emph{The twisted twin of the
  {G}rigorchuk group}, Internat. J. Algebra Comput. \textbf{20} (2010), no.~4,
  419--450. \MR{2891709}

\bibitem[BSZ12]{BSZ12}
Laurent Bartholdi, Olivier Siegenthaler, and Pavel Zalesskii, \emph{The
  congruence subgroup problem for branch groups}, Israel J. Math. \textbf{187}
  (2012), 419--450.

\bibitem[dlH00]{MR1786869}
Pierre de~la Harpe, \emph{Topics in geometric group theory}, Chicago Lectures
  in Mathematics, University of Chicago Press, Chicago, IL, 2000. \MR{1786869}

\bibitem[GNS01]{GNS01}
Piotr~W. Gawron, Volodymyr~V. Nekrashevych, and Vitaly~I. Sushchansky,
  \emph{Conjugation in tree automorphism groups}, Internat. J. Algebra Comput.
  \textbf{11} (2001), no.~5, 529--547. \MR{1869230}

\bibitem[Gri84]{MR764305}
R.~I. Grigorchuk, \emph{Degrees of growth of finitely generated groups and the
  theory of invariant means}, Izv. Akad. Nauk SSSR Ser. Mat. \textbf{48}
  (1984), no.~5, 939--985. \MR{764305}

\bibitem[Gri00]{MR1765119}
\bysame, \emph{Just infinite branch groups}, New horizons in pro-{$p$} groups,
  Progr. Math., vol. 184, Birkh\"auser Boston, Boston, MA, 2000, pp.~121--179.
  \MR{1765119}

\bibitem[GS83]{MR759409}
N.~Gupta and Said Sidki, \emph{Some infinite {$p$}-groups}, Algebra i Logika
  \textbf{22} (1983), no.~5, 584--589. \MR{759409}

\bibitem[GS84]{MR767112}
Narain Gupta and Said Sidki, \emph{Extension of groups by tree automorphisms},
  Contributions to group theory, Contemp. Math., vol.~33, Amer. Math. Soc.,
  Providence, RI, 1984, pp.~232--246. \MR{767112}

\bibitem[GW14]{MR3119213}
Alejandra Garrido and John~S. Wilson, \emph{On subgroups of finite index in
  branch groups}, J. Algebra \textbf{397} (2014), 32--38. \MR{3119213}

\bibitem[LS03]{MR1978431}
Alexander Lubotzky and Dan Segal, \emph{Subgroup growth}, Progress in
  Mathematics, vol. 212, Birkh\"auser Verlag, Basel, 2003. \MR{1978431}

\bibitem[Per07]{per07}
Ekaterina Pervova, \emph{Profinite completions of some groups acting on trees},
  J. Algebra \textbf{310} (2007), no.~2, 858--879. \MR{2308183}

\bibitem[R{\"{o}}v02]{rov02}
Claas~E. R{\"{o}}ver, \emph{Abstract commensurators of groups acting on rooted
  trees}, Proceedings of the {C}onference on {G}eometric and {C}ombinatorial
  {G}roup {T}heory, {P}art {I} ({H}aifa, 2000), vol.~94, 2002, pp.~45--61.
  \MR{1950873}

\end{thebibliography}
\end{document}